\documentclass[12pt]{amsart}
\usepackage[usenames]{color} 
\usepackage{amsmath,amsfonts,amsthm}   
\usepackage[top=1in, bottom=1in, left=1.3in, right=1.3in]{geometry}
\usepackage{mathrsfs} 
\usepackage{mathabx}
\usepackage[colorlinks=true, linkcolor=blue, urlcolor=blue, pagebackref=true]{hyperref}
\usepackage{syntonly} 
\usepackage{stmaryrd}
\usepackage[mathcal]{euscript} 
\usepackage{setspace}

\usepackage[T1]{fontenc}

\newcommand{\enote}[1]{#1}

\renewcommand{\th}{\text{th}}

\newcommand{\set}[1]{\left\{#1\right\}}
\renewcommand{\(}{\left(}
\renewcommand{\)}{\right)}

\newcommand{\up}[1]{\left\lceil #1 \right\rceil}
\newcommand{\down}[1]{\left\lfloor #1 \right\rfloor}
\newcommand{\tr}[2]{\left \langle {#1} \right \rangle_{#2}}

\newcommand{\NN}{\mathbb{N}}

\newcommand{\FF}{\mathbb{F}}
\newcommand{\QQ}{\mathbb{Q}}
\newcommand{\RR}{\mathbb{R}}

\newcommand{\m}{\mathfrak{m}}

\newcommand{\Length}{\operatorname{length}}

\newcommand{\Jac}{\hbox{\rm{Jac}}}

\newcommand{\fpt}[1]{\boldsymbol{\operatorname{fpt}}(#1)}

\newcommand{\bracket}[2]{{#1}^{[p^{#2}]}}

\theoremstyle{definition}
\newtheorem{theorem}{Theorem}[section]
\newtheorem*{MainThm}{Theorem}

\newtheorem{corollary}[theorem]{Corollary}
\newtheorem{lemma}[theorem]{Lemma}
\newtheorem{proposition}[theorem]{Proposition}

\newtheorem{notation}[theorem]{Notation}

\theoremstyle{definition}
\newtheorem{definition}[theorem]{Definition}

\newtheorem{example}[theorem]{Example}
\newtheorem{conjecture}[theorem]{Conjecture}
\newtheorem{alg}[theorem]{Algorithm}
\newtheorem{remark}[theorem]{Remark}
\numberwithin{equation}{subsection}
 
\newcommand{\ft}[2]{\boldsymbol{\operatorname{c}}^{#2}(#1)}
\newcommand{\jnset}[1]{\mathbb{J}(#1)}

\newcommand{\testR}[3]{\tau(#1 \, ,{#2}^{#3})}
\newcommand{\test}[2]{\tau({#1}^{#2})}
\newcommand{\displaytest}[2]{\tau\({#1}^{#2}\)}
\newcommand{\size}[1]{\vphantom{-}^{\#}#1}
\newcommand{\error}{\varepsilon}
\newcommand{\ideala}{\mathfrak{a}}
\newcommand{\idealb}{\mathfrak{b}}
\renewcommand{\bracket}[2]{{#1}^{\left[ p^{#2} \right]}}
\newcommand{\ibracket}[2]{{\(#1\)}^{\left[1/{p^{#2}} \right]}}

 \newcommand{\quot}[2]{ {#1} \hspace{.01in} / \hspace{.025in} {#2} }

 \newcommand{\dtf}[1]{{\slshape #1}}

\AtBeginDocument{%
   \def\MR#1{}
}

\begin{document}
\title[Local $\m$-adic constancy of $F$-pure thresholds and test ideals]{Local $\m$-adic constancy of $F$-pure thresholds \\ and test ideals }
\author{Daniel J.\- Hern\'andez}
\author{ Luis N\'u\~nez-Betancourt}
\author{Emily E.\- Witt}
\maketitle

\begin{abstract}
In this note, we consider a corollary of the ACC conjecture for $F$-pure thresholds.  Specifically, we show that the $F$-pure threshold (and more generally, the test ideals) associated to a polynomial with an isolated singularity are locally constant in the $\m$-adic topology of the corresponding local ring.  As a by-product of our methods, we also describe a  simple algorithm for computing all of the $F$-jumping numbers and test ideals associated to an arbitrary polynomial over an $F$-finite field.
\end{abstract}

\section{Introduction}
Given a regular ring $R$ of prime characteristic, the $F$-pure threshold of an element $f$ of $R$, denoted $\fpt{f}$, is a numerical invariant defined via the Frobenius endomorphism, and may be thought of as a measure of how ``far'' the hypersurface $\mathbb{V}(f) \subseteq \operatorname{Spec}(R)$ is from being $F$-pure \cite{TW2004}.  In this article, we will often focus on the concrete case that $R$ is a (localization of) a polynomial ring over $\mathbb{F}_p$.  In this restricted setting, the $F$-pure threshold may also be considered as the prime characteristic analog of the so-called log canonical threshold, an important invariant measuring the singularities of hypersurfaces over fields of characteristic zero (see, e.g., the survey \cite{BFS2013}).  Motivated by these connections, it is expected that $F$-pure and log canonical thresholds satisfy similar properties.  

One particularly important property of log canonical thresholds is that they satisfy the ascending chain condition (ACC).   In the simplest context, this condition states that given any sequence of polynomials in $\mathbb{C}[x_1, \cdots, x_n]$, the associated sequence of log canonical thresholds contains no strictly increasing subsequence.  The assertion that log canonical thresholds satisfy ACC was first conjectured by Shokurov \cite{ShokurovACC}, and was shown to hold in many cases (see, e.g., \cite{dfEM2010, dfEM2011}) before being completely solved in \cite{ACCLCT}.  We now recall the characteristic $p$ analogue of this conjecture \cite[Conjecture 4.4]{BMS-Hyp}.

\begin{conjecture}[ACC for $F$-pure thresholds]\label{Conj ACC}
Given any regular ring $R$ of prime characteristic, and any sequence of elements of $R$, \enote{the associated sequence of $F$-pure thresholds is not strictly increasing.}
\end{conjecture}

Conjecture \ref{Conj ACC} is essentially open (even when $R$ is a polynomial ring), and is expected to be difficult to solve; for a positive result corresponding to a very specialized case, see \cite[Proposition 7.3]{QuasiHomog}. In what follows, we note one particular consequence of Conjecture \ref{Conj ACC}:  Fix an element $f$ of a regular local ring $(R, \m)$ of prime characteristic.  Given arbitrary elements $h_k \in \m^k$, it is easy to see (e.g., see \cite[Corollary 3.4]{BMS-Hyp}) that $\fpt{f+ h_k} \to \fpt{f}$ as $k \to \infty$.  Assuming Conjecture \ref{Conj ACC}, the terms in this sequence must be eventually greater than or equal to $\fpt{f}$\footnote{This observation was first recorded in \cite[Remark 4.5]{BMS-Hyp}. Note that the (incorrect) inequality $\fpt{f} \geq \fpt{f+h}$ appearing therein should be reversed.} (since otherwise, there would be infinitely many terms strictly less than $\fpt{f}$, and hence, there would exist a strictly increasing subsequence converging to $\fpt{f}$).  In summary, Conjecture \ref{Conj ACC} predicts the following.

\begin{conjecture}
\label{semicont: C}
If $(R,\m)$ is a regular local ring of prime characteristic, and $f$ is any element of $\m$, then there exists an integer $N>0$ such that $\fpt{f} \leq \fpt {f + h}$ for every  $h \in \m^N$. 
\end{conjecture}


Conjecture \ref{semicont: C} is known to be true for the image of a homogeneous polynomial $f$ inside of the local ring at the origin (in this case, the integer $N=\deg(f) + 1$ suffices) \cite[Proposition 7.4]{QuasiHomog}. However, in the absence of such homogeneity, verifying Conjecture \ref{semicont: C} appears to be a much more subtle issue.  

On a related note, we briefly switch gears, recalling the notion of test ideals.  Given a regular ring $R$ of prime characteristic and an element $f \in R$, one may use the Frobenius endomorphism to construct a family of ideals $\test{f}{\lambda}$ of $R$, parameterized by a non-negative real parameter $\lambda$, called the (generalized) {test ideals} of $f$.  For more on test ideals, the reader may refer to \cite{H-Y, BMS-MMJ}, as well as to the many references (e.g., \cite{HoHu1}) cited therein.  The relevance of test ideals to the present discussion stems from the fact that $\fpt{f}$ may be described in terms of the behavior of $\test{f}{\lambda}$ as the parameter $\lambda$ varies; specifically,  the $F$-pure threshold is the smallest positive parameter $\lambda$ for which $\test{f}{\lambda}$ is not the trivial ideal (see, e.g., Remark \ref{FPT: Rmk}).  More generally, as $\lambda$ varies, the values of the test ideals $\test{f}{\lambda}$ allow us to define a sequence of numerical invariants called $F$-jumping numbers, and the $F$-pure threshold is the smallest positive member of this family.  

In the main results of this paper, Theorems \ref{ThmSC} and \ref{ThmSC-Ext}, we show that stronger versions of Conjecture \ref{semicont: C} (i.e., analogous statements involving $F$-pure thresholds, test ideals, and higher jumping numbers) hold whenever $f$ is a polynomial with an isolated singularity.  Recall that a field $\mathbb{L}$ of characteristic $p>0$ is called $F$-finite if $\mathbb{L}^p \subseteq \mathbb{L}$ is a finite field extension.

%
%



\begin{MainThm}[cf.\ Theorems \ref{ThmSC} and \ref{ThmSC-Ext}] 
Let $R$ be a polynomial ring over an $F$-finite field of prime characteristic $p>0$, and let $\m$ denote the maximal ideal of $R$ generated by the variables.  If $f \in \m$ is such that $\Jac(f)$, the ideal of $R$ generated by $f$ and all of its partial derivatives, is primary to $\m$, then there exist integers $N < M$ with the following properties:
\begin{enumerate}
\item After localizing at the origin, $\fpt{f}= \fpt{f+h}$ for every $h \in \m^N$.
\item After localizing at the origin, $\test{f}{\lambda} = \test{(f+h)}{\lambda} $ for every $h \in \m^M$ and every $0 \leq \lambda < 1$.
In particular, after localizing at the origin, $f$ and $f+h$ have the same $F$-jumping numbers.
\end{enumerate} 
\end{MainThm}

\begin{remark}
We note that values for $N$ and $M$ in the preceding Theorem are explicitly given in terms of the characteristic and the length of $\quot{R}{\Jac(f)}$ (which is finite, by hypothesis).
\end{remark}

As a by-product of the methods used to prove our main results,  we produce explicit algorithms for computing the $F$-pure threshold of an arbitrary polynomial $f$ (and more generally, all the higher $F$-jumping numbers associated to $f$), as well as the test ideals $\test{f}{\lambda}$, as $\lambda$ varies over all non-negative real numbers (cf.\ Algorithms \ref{AlgIsoSingTestIdeal} and \ref{AlgIsoSingTodo}).  Though it is possible that (some aspects of) these algorithms may already be known to experts, they remain interesting, mostly due to their simplicity.  Indeed, these algorithms are constructed by appealing only to first principles (e.g., see Proposition \ref{ComputationTestIdeal:P}), and, in particular, avoid the use of heavy machinery.  
We would like to highlight the fact that other algorithms have also been designed to compute test ideals and related objects (see, e.g., \cite{Alg08, Alg12, Alg14}).

\begin{notation}
Throughout this article, $p$ will denote a prime integer, and $\up{\rho}$ will denote the least integer greater than or equal to a real number $\rho$.  Additionally, if $R$ is a polynomial ring over a field, and $f \in R$ is an arbitrary polynomial, then  $\Jac(f)$ will denote the Jacobian ideal of $f$, i.e., the ideal of $R$ generated by $f$ and all of its partial derivatives.
\end{notation}

\section{Preliminaries}

\newcommand{\fact}[2]{p^{#1}(p^{#2}-1)}
\newcommand{\expset}[2]{\mathscr{E}_{#2}\( #1 \)}
\newcommand{\mexp}[2]{\epsilon_{#2}\(#1\)}

\subsection{Representations of rational numbers}

\begin{definition}
\label{trunc: D}
Given $\lambda \in \RR_+$ and $e \in \NN$, we call $\tr{\lambda}{e}:=\frac{\up{p^e \lambda}-1}{p^e}$ the \dtf{$e^{\th}$ truncation of $\lambda$ \textup{(}base $p$\textup{)}}.
\end{definition}

This terminology is justified by the following observation:  $\tr{\lambda}{e}$ is the $e^{\th}$ truncation of the (unique) non-terminating base $p$ expansion of $\lambda$; i.e., if $\lambda = \sum_{s=1}^N a_s \cdot p^s + \sum_{s=1}^{\infty} {b_s} \cdot p^{-s}$ is the non-terminating base $p$ expansion of $\lambda$, then $\tr{\lambda}{e} = \sum_{s=1}^N a_s \cdot p^s + \sum_{s=1}^{e} {b_s} \cdot p^{-s}$.

\begin{remark}
\label{obvious: R}
The preceding interpretation is often useful; e.g., from this perspective, the following is clear:  Given $\lambda \in \RR_+$, the truncations $\tr{\lambda}{e}$ form a non-decreasing sequence strictly less than, and converging to, $\lambda$.  Moreover, if $\gamma \in \RR_+$ and $\tr{\lambda}{e} = \tr{\gamma}{e}$, then $\tr{\lambda}{s} = \tr{\gamma}{s}$ for every $0 \leq s \leq e$.
\end{remark}


\begin{definition}  
\label{expset: D}
Given $\lambda \in \QQ_+$, set $\expset{\lambda}{p} := \set{  (u,v)  \in \NN\times \NN_+:  \fact{u}{v} \cdot \lambda \in \NN}$.
\end{definition}

\begin{remark}
\label{mexp: R}
Note that $\expset{\lambda}{p}$ is non-empty: Writing $\lambda = p^{\nu} \cdot \frac{m}{n}$, where $\nu \in \mathbb{Z}$ and $m$, $n \in \NN$, and $p$, $m$, and $n$ are pairwise relatively prime, we see that \[\( \min \set{ -\nu, 0}, \min \set{ s \geq 1 : p^s \equiv 1 \bmod n} \) \in \expset{\lambda}{p}.\]  
\end{remark}

\begin{definition}
\label{mexp: D}
We use $\mexp{\lambda}{p}$ to denote the canonical element of $\expset{\lambda}{p}$ described in Remark \ref{mexp: R}.
\end{definition}

\begin{remark}
\label{factor: R}  
The fact that $p^v-1$ divides $p^{vb}-1$ for every $b \geq 1$ implies that $(u+a,vb) \in \expset{\lambda}{p}$ for all $(u,v) \in \expset{\lambda}{p}$ and $(a,b) \in \NN \times \NN_+$.  Moreover, a straightforward computation shows that $\mexp{\lambda}{p}$ generates $\expset{\lambda}{p}$, in  the sense that if $\mexp{\lambda}{p} = ({u}, {v})$, then $\expset{\lambda}{p} = \set{ ({u} + a, {v} b ) : (a,b) \in \NN \times \NN_+}$.
\end{remark}

\begin{corollary}
\label{distinctElements: C}
 Fix $\lambda \in \QQ_+$.  If $\mexp{\lambda}{p} = (u,v)$, then $\size{\set{ p^e \lambda - \down{p^e \lambda} : 0 \leq e \leq u+v-1}} = u+v$.
\end{corollary}

\begin{proof}  The assertion is clear when $(u,v) = (0,1)$.  Next, assume that $(u,v) \neq (0,1)$, and by means of contradiction, suppose there exists $0 \leq a < b \leq u+v-1$ such that $p^a \lambda - \down{p^a \lambda} = p^b \lambda - \down{p^b \lambda}$.  Consequently, $ p^b \lambda - p^a \lambda  = p^a (p^{b-a} - 1) \lambda \in \NN$, and since $a<b$, it follows that $(a,b-a) \in \expset{\lambda}{p}$. However, as $a+(b-a) = b \leq u+v-1$, it follows that either $a \leq u$ or $b-a \leq v$, which is impossible by Remark \ref{factor: R}.
\end{proof}

%
%

\begin{lemma}
\label{agreement: L}
If $(u,v) \in \expset{\lambda}{p}$ and $(a,b) \in \expset{\gamma}{p}$, then $\lambda = \gamma$ if and only if $\tr{\lambda}{u+a + vb} = \tr{\gamma}{u+a+vb}$.
\end{lemma}

\begin{proof}
The forward implication is clear, so we focus on the reverse one.  Setting $(s,e) = (u+a,vb)$, it follows from Remark \ref{factor: R} that $(s,e) \in \expset{\lambda}{p} \cap \expset{\gamma}{p}$.  By Remark \ref{obvious: R}, our hypothesis that $\tr{\lambda}{s+e} = \tr{\gamma}{s+e}$ implies that $\tr{\lambda}{s} = \tr{\gamma}{s}$ and hence, that $\up{p^s \lambda} = \up{p^s \gamma}$.  In what follows, $N$ will denote this common value.  

As $p^{s+e} \lambda = p^s(p^e-1) \lambda + p^s \lambda$, we see that $\up{p^{s+e} \lambda} = p^s(p^e-1) \lambda + N$; by symmetry, we also have that $\up{p^{s+e} \gamma} = p^s(p^e-1) \gamma + N$.  Finally, the assumption that $\tr{\lambda}{s+e} = \tr{\gamma}{s+e}$ implies that $\up{p^{s+e}\lambda} = \up{p^{s+e}\gamma}$, and it follows that $\lambda = \gamma$.
\end{proof}

\begin{corollary}
\label{sumOfExponents: C}
Fix $\lambda \in \QQ_+$ and $(u,v) \in \expset{\lambda}{p}$.    If  $w \in \NN_+$,  and $\gamma \in \QQ_+$ is contained in \[ \( \tr{\lambda}{u+vw}, \lambda \) \cup \( \lambda, \tr{\lambda}{u+vw} + {p^{-u-vw}} \right], \] then $a+b > w$ for every $(a,b) \in \expset{\gamma}{p}$.
\end{corollary}

\begin{proof} As the reader may readily verify, the assumption that $\gamma$ is contained in either of these intervals implies that $\tr{\gamma}{u+vw} = \tr{\lambda}{u+vw}$.  By means of contradiction, suppose that $a+b \leq w$ for some $(a,b) \in \expset{\gamma}{p}$.  As $v \in \NN_+$, by definition of $\expset{\lambda}{p}$, it follows that $u+vw \geq u+v(a+b) \geq u+a + vb$, and Remark~\ref{obvious: R} then implies that $\tr{\gamma}{u+a+vb} = \tr{\lambda}{u+a+vb}$.  Finally, Lemma~\ref{agreement: L} allows us to conclude that $\gamma = \lambda$, which is a contradiction.
\end{proof}

\subsection{Generalized test ideals}\label{SecGTI}
Test ideals were introduced and studied by Hochster and Huneke (see, e.g.,  \cite{HoHu1,HoHu2}), and were later generalized to the context of pairs by Hara and Yoshida \cite{H-Y}.  Below, we recall the basic properties of these so-called generalized test ideals.  In this article, we will always work in a regular ambient ring, and will hence follow the concrete description of these ideals given by Blickle, Musta\c{t}\u{a}, and Smith \cite{BMS-MMJ}.

Throughout this subsection, $R$ will denote a regular ring of characteristic $p>0$, and $f$ will denote an arbitrary element of $R$.  Given $e \in \NN_+$, we will use $\ibracket{f}{e}$ to denote the smallest ideal $\idealb$ of $R$ (with respect to inclusion) with the property that $f \in \bracket{\idealb}{e}$;  the fact that such an ideal exists is a consequence of the flatness of the Frobenius map on $R$ \cite[Definition $2.2$]{BMS-MMJ}.  

\begin{remark}
\label{ibracketcomputation: R}  When $R = \mathbb{L}[x_1, \cdots, x_n]$ is a polynomial ring over a field $\mathbb{L}$ with $[\mathbb{L}:\mathbb{L}^p] < \infty$, the ideal $\ibracket{f}{e}$ may be described as follows:  If $\mathbb{B}_e$ is any free basis for $R$ over the subring $R^{p^e} = \mathbb{L}^{p^e}[x_1^{p^e}, \cdots, x_n^{p^e}]$ and $f = \sum_{\mu \in \mathbb{B}_e}  f_{\mu}^{p^e} \cdot \mu$ is the representation of $f$ in this basis, then $\ibracket{f}{e}$ is the ideal of $R$ generated by the coordinates $\set{ f_{\mu} : \mu \in \mathbb{B}_e}$.  This ideal does not depend on the choice of basis for $R$ over $R^{p^e}$; see, e.g., \cite[Section 3]{AMBL} or \cite[Proposition 2.5]{BMS-MMJ}. 
 \end{remark}

\begin{definition}
\label{TI: D}
  The ideals $\ibracket{f^{\up{p^e \lambda}}}{e}$ form an ascending chain, and hence, must stabilize. We call the stable member of this chain the \dtf{test ideal of $f$ with respect to the parameter $\lambda$}, and denote it by $\testR{R}{f}{\lambda}$.  Restated, $\testR{R}{f}{\lambda} = \bigcup_{e \geq 1} \ibracket{f^{\up{p^e \lambda}}}{e}$, which equals $\ibracket{f^{\up{p^e \lambda}}}{e}$ for $e \gg 0$.  When the ambient ring is clear from context, we will often write $\test{f}{\lambda}$ instead of $\testR{R}{f}{\lambda}$.
\end{definition}

\begin{remark}
\label{stabilizationTI: R}
If $p^e \lambda \in \NN$ for some $e \geq 1$, then the chain of ideals defining $\test{f}{\lambda}$ stabilizes at the $e^{\th}$ step; in other words, $\test{f}{\lambda} = \ibracket{f^{\up{p^e \lambda}}}{e}$ \cite[Lemma $2.1$]{BMS-Hyp}.
\end{remark}

\begin{remark}
\label{varyingParameter: R}
Here, we recall some of the ways in which $\test{f}{\lambda}$ varies with $\lambda$.
\begin{enumerate}
\item \label{textIdealsDecrease: i} $\test{f}{\lambda} \subseteq \test{f}{\gamma}$ whenever $\lambda \geq \gamma$ \cite[Proposition $2.11$]{BMS-MMJ}.
\item \label{constanttotheright: i} Given $\lambda \geq 0$, $\exists \ \error > 0$ such that $\test{f}{\lambda} = \test{f}{\lambda + \delta}$ for all $0 \leq \delta \leq \error$ \cite[Corollary $2.16$]{BMS-MMJ}.
\item \label{Skoda: i} $\test{f}{\lambda} = f^{\down{\lambda}} \cdot \test{f}{\lambda-\down{\lambda}}$ \cite[Theorem 2.25]{BMS-MMJ}.  In particular, the set of all test ideals $\set{\test{f}{\lambda} : \lambda \in \RR_{\geq 0}}$  is completely determined by $\set{\test{f}{\lambda}: 0 \leq \lambda < 1}$. 
\end{enumerate}
\end{remark}

\begin{definition}  A real number $\xi \geq 0$ is called an \dtf{$F$-jumping number of $(R,f)$} if $\xi = 0$, or $\xi > 0$ and $\test{f}{\xi} \neq \test{f}{\xi-\error}$ for every $0 < \error < \xi$.  We will use $\jnset{R,f}$, and often simply $\jnset{f}$, to denote the set of all $F$-jumping numbers of $f$. 
\end{definition}

\begin{remark}[$F$-pure threshold]  
\label{FPT: Rmk}
As the test ideal corresponding to the parameter $\lambda = 0$ is trivial, Remark \ref{varyingParameter: R} \eqref{constanttotheright: i} implies that $\test{f}{\lambda}$ is trivial for small positive values of $\lambda$, and hence, the first positive $F$-jumping number of $(R,f)$ may be described as $\sup \set{ \lambda > 0 : \test{f}{\lambda} = R}$.  We call this $F$-jumping number the \dtf{$F$-pure threshold of $(R,f)$}, and denote it by $\fpt{R,f}$.
\end{remark}

\begin{theorem}
\label{discretenessRationality: T} \cite[Theorem $1.1$]{BMS-Hyp} The $F$-jumping numbers of $f$ are discrete and rational:  for every bounded interval $I \subseteq \RR_{\geq 0}$, we have that $\jnset{f} \cap I$ is a finite subset of $\QQ$.
\end{theorem}

\begin{remark}
\label{bijectionJNTI: R}
By definition, given any bounded interval $I$ of $\mathbb{R}_{\geq 0}$, the map $\lambda \mapsto \test{f}{\lambda}$ induces a bijection between $\jnset{f} \cap I$ and $\set{ \test{f}{\lambda}: \lambda \in I}$. 
\end{remark}


\begin{lemma}
\label{SelfMap: L}
For every $e \in \NN_+$, the function $\lambda \mapsto p^e \lambda - \down{p^e \lambda}$ maps $\jnset{f}$ to $\jnset{f} \cap [0,1)$.
\end{lemma}

\begin{proof} As this fact is well-known (see, e.g., \cite[Proposition 3.4]{BMS-MMJ}), we will only provide a sketch of the proof: Given $\lambda \in \jnset{f}$, one may show that $\test{f}{p^e  \lambda} \subseteq \bracket{\test{f}{\lambda}}{e}$, while $\test{f}{{p^e  \lambda}-\error} \not \subseteq \bracket{\test{f}{\lambda}}{e}$ for every $0 < \error < p^e \lambda$, which shows $p^e \lambda \in \jnset{f}$.  Finally, Remark \ref{varyingParameter: R} \eqref{Skoda: i} then shows that $p^e \lambda - \down{p^e \lambda} \in \jnset{f}$, as well.
\end{proof}

\begin{proposition} \cite[Proposition 2.13]{BMS-MMJ}  
\label{commutesWithLocalization: P}  The formation of the test ideal commutes with localization.  In particular, if $\m$ is a prime ideal of $R$, then $\testR{R}{f}{\lambda} \cdot R_{\m} = \testR{R_{\m}}{f}{\lambda}$.
\end{proposition}

\subsection{$F$-thresholds}
In this subsection, we discuss some basic properties of $F$-thresholds.

\begin{definition}\cite{MTW2005}  If $f$ is an element of an $F$-finite regular ring $R$ and $\idealb$ is a proper ideal of $R$ with  $f \in \sqrt{\idealb}$,  the limit $\ft{f}{\idealb} := \lim_{e \to \infty} p^{-e} \cdot {\max \set{ N \geq 1 : f^N \notin \bracket{\idealb}{e}}}$ exists, and is called the \dtf{$F$-threshold of $f$ with respect to $\idealb$}.  If $\idealb = R$, we set $\ft{f}{\idealb}= 0$.
\end{definition}

Our interest in $F$-thresholds stems from the following result.

\begin{proposition} \cite[Proposition 2.29 and Corollary 2.30]{BMS-MMJ}
\label{F-thresholds: P}
Fix an element $f$ of an $F$-finite regular ring $R$.
\begin{enumerate}
\item \label{Fthres1: i} If $\idealb$ is an ideal of $R$ with $f \in \sqrt{\idealb}$, then $\test{f}{\ft{f}{\idealb}} \subseteq \idealb$. 
\item \label{Fthres2: i} If $\lambda \in \RR_{\geq 0}$, then $f \in \sqrt{\test{f}{\lambda}}$, and $\ft{f}{\test{f}{\lambda}} \leq \lambda$.
\end{enumerate}

Furthermore, the $F$-jumping numbers of $f$ coincide with the $F$-thresholds of $f$.  More precisely, \[ \jnset{R,f} = \set{ \ft{f}{\idealb} : f \in \sqrt{\idealb} }.\]
\end{proposition}


\begin{remark}[The $F$-pure threshold realized as an $F$-threshold]
\label{FPTasFT: Rmk}
In the case that $(R,\m)$ is local, the description of $\fpt{f}$ given in Remark \ref{FPT: Rmk} becomes $\fpt{f} = \sup \set{\lambda > 0 : \test{f}{\lambda} \subseteq \m}$, and it follows (e.g., from an argument based on Proposition \ref{F-thresholds: P}) that $\fpt{f} = \ft{f}{\m}$.
\end{remark}

\begin{lemma}\textup{\cite[Lemma 3.3 and Corollary 3.4]{BMS-Hyp}} \label{Lem Ineq} Let $R$ be an $F$-finite regular ring of characteristic $p>0$.  If $\m$ is a maximal ideal of $R$, $\idealb$ is an ideal of $R$ such that $\m^{[p^e]} \subseteq \idealb \subseteq \m$ for some  $e \geq 0,$ and $f, g \in \m$ are such that $f - g \in \m^N$ for some $N>0$, then
\[|\ft{f}{\idealb} - \ft{g}{\idealb} | \leq \frac{p^e \dim(R)}{N}.\]
\end{lemma}

\section{Computations}
In this section, we describe simple algorithms for computing test ideals and jumping numbers.

\begin{remark}[Bounds for the number of $F$-jumping numbers] \label{RemBounds: R}
Fix a polynomial ring $R$ over an $F$-finite field $\mathbb{L}$, and a polynomial $f \in R$.  As described in 
\enote{Theorem \ref{discretenessRationality: T},}
the number of $F$-jumping numbers of $f$ contained in $[0,1)$ is always finite, and hence bounded above by an integer $B$.  Such bounds will play a key role in our algorithms, and below, we recall some possible choices for $B$.

\begin{enumerate}
\item \label{bounds2} If $\lambda \in [0,1)$, then $\tau(f^{\lambda})$ is generated by polynomials in $\left[R\right]_{\leq \deg(f)}$ (i.e., by polynomials of degree at most $\deg(f)$) \cite[Proposition $3.2$]{BMS-MMJ}.  Consequently, Remark \ref{bijectionJNTI: R} implies that \[ \size{ \jnset{R,f} \cap [0,1) } = \size{ \set{ \test{f}{\lambda} : 0 \leq \lambda < 1 }} \leq \dim_{\mathbb{L}} \left[R\right]_{\leq \deg f} = \binom{\dim R+\deg f}{\dim R}.\] 

\item \label{bounds3} If $\lambda \in[0,1)$, then $\Jac(f) \subseteq \test{f}{\lambda}$ \cite[Theorem 1.3 and Corollary 1.4]{KLZ}.  If $\Jac(f)$ happens to be primary to the homogeneous maximal ideal of $R$, then Remark \ref{bijectionJNTI: R} shows that $\size{ \jnset{R,f} \cap [0,1) }  \leq \dim_{\mathbb{L}} \quot{R}{\Jac(f)} = \Length_R \quot{R}{\Jac(f)}$.
\end{enumerate}
\end{remark}

Proposition \ref{Shape} below, which may be thought of as a generalization of \cite[Proposition 3.8]{BMS-MMJ},  forms the basis for Algorithms \ref{AlgIsoSingTestIdeal} and \ref{AlgIsoSingTodo}, and also plays a key role in the proof of Theorems \ref{ThmSC} and \ref{ThmSC-Ext} in Section \ref{madic: S}.

\begin{proposition}
\label{Shape}  If $R$ is an $F$-finite regular ring, $f \in R$ any element, and $\size{ \jnset{f} \cap [0,1)} \leq B$, then 
\[ \jnset{f} \subseteq \mathscr{A}_B:=\set{ \lambda \in \QQ_+: \text{ there exists} \ (a,b) \in \expset{\lambda}{p} \text{ with } a+b \leq B} \cup \set{ 0 }.  \] 
\end{proposition}

\begin{proof}
If $0 \neq \lambda \in \jnset{f}$ and $\mexp{\lambda}{p}  = (u,v)$, then Corollary \ref{distinctElements: C} and Lemma \ref{SelfMap: L} show that \[\set{ p^e \lambda - \down{p^e \lambda} : 0 \leq e \leq u+v-1}\] is a subset of $\jnset{f} \cap [0,1)$ of cardinality $u+v$, so that $u+v \leq \size{\jnset{f} \cap [0,1)} \leq B$.
\end{proof}

\subsection{Stabilization of chains of test ideals}  In this subsection, we identify when two chains of test ideals (one increasing, the other decreasing) stabilize. 

\begin{remark} \label{stablizationUnion: R}
As $\up{p^e \lambda} = p^e \tr{\lambda}{e} + 1$, by definition, Definition \ref{TI: D} and Remark \ref{stabilizationTI: R} imply that \[\test{f}{\lambda} = \bigcup_{e \geq 1} \ibracket{f^{\up{p^e \lambda}}}{e} = \bigcup_{e \geq 1} \ibracket{f^{p^e \tr{\lambda}{e}+1}}{e} = \bigcup_{e \geq 1} \test{f}{\tr{\lambda}{e} + p^{-e}},\]
which equals $\test{f}{\tr{\lambda}{e} + p^{-e}}$  for $e \gg 0$. 
\end{remark}

\begin{remark} \label{stablizationIntersection: R}
By Remark \ref{varyingParameter: R} \eqref{textIdealsDecrease: i}, the test ideals $\test{f}{\lambda-\error}$ decrease as $\error \searrow 0$.  Moreover, when considering the intersection of these nested ideals, by the same remark, we can replace the sequence $\lambda - \error$ as $\error \searrow 0$ with \emph{any} sequence converging to $\lambda$ from below, so that $\bigcap_{ 0 < \varepsilon < \lambda} \tau(f^{\lambda-\varepsilon}) = \bigcap_{e \geq 1} \test{f}{\tr{\lambda}{e}}.$
Theorem \ref{discretenessRationality: T} implies that there are only finitely many different test ideals as the parameter varies within the interval $[0,\lambda)$, and so this intersection is eventually stable;  i.e., 
\[  \bigcap_{ 0 < \varepsilon < \lambda} \tau(f^{\lambda-\varepsilon}) = \bigcap_{e \geq 1} \test{f}{\tr{\lambda}{e}} = \test{f}{\tr{\lambda}{e}} \text{ for $e \gg 0$}. \] 
\end{remark}
 
In the following proposition, we find points at which the sequences of ideals in the union in Remark \ref{stablizationUnion: R} and in the intersection on Remark \ref{stablizationIntersection: R} must stabilize.

\renewcommand{\stop}{s}

\begin{proposition}\label{ComputationTestIdeal:P}
 Fix any element $f$ of an $F$-finite regular ring, and a positive integer $B$ for which $\size{ \jnset{f} \cap [0,1)} \leq B$.  Given $\lambda \in \QQ_+$ and $(u,v) \in \expset{\lambda}{p}$, if $\stop=u+vB$, then \[ (\tr{\lambda}{\stop}, \lambda) \cup \(\lambda, \tr{\lambda}{\stop} + {p^{-\stop}} \ \right] \]  contains no element of $\jnset{f}$.  In particular, $\bigcap \limits_{ 0 < \varepsilon < \lambda} \tau(f^{\lambda-\varepsilon}) = \displaytest{f}{\tr{\lambda}{\stop}} \text{ and } \test{f}{\lambda}   =  \displaytest{f}{\tr{\lambda}{\stop}+p^{-\stop}}$.
\end{proposition}

\begin{proof} If $\gamma \in \jnset{f}$ is contained in either of these intervals, then setting $w=B$ in Corollary \ref{sumOfExponents: C} implies that $a+b > B$ for every $(a,b) \in \expset{\gamma}{p}$, which directly contradicts Proposition \ref{Shape}.

Regarding $\bigcap_{0 < \error < \lambda} \tau(f^{\lambda-\varepsilon})$: As $(\tr{\lambda}{\stop}, \lambda)$ contains no $F$-jumping numbers, $\test{f}{\gamma} = \test{f}{\tr{\lambda}{\stop}}$ for every $\gamma \in (\tr{\lambda}{\stop}, \lambda)$, and it follows that $\bigcap_{ 0 < \varepsilon < \lambda} \tau(f^{\lambda-\varepsilon}) = \bigcap_{e \geq 1} \displaytest{f}{\tr{\lambda}{e}} = \displaytest{f}{\tr{\lambda}{\stop}}$.

Regarding $\test{f}{\lambda}$:  If $p^{\stop} \lambda \in \NN$, then $\lambda = \tr{\lambda}{\stop} + p^{-\stop}$, and our assertion is trivial.  If, instead, $p^{\stop} \lambda \notin \NN$, the (non-empty) interval $(\lambda, \tr{\lambda}{\stop} + p^{-\stop}]$ contains no $F$-jumping numbers, and hence $\test{f}{\lambda} = \test{f}{\gamma}$ for every 
$\gamma$ in this interval;  our assertion then follows by choosing $\gamma = \tr{\lambda}{\stop} + p^{-\stop}$.
\end{proof}

\subsection{Algorithms}
Throughout this subsection, we fix an element $f$ of an $F$-finite polynomial ring, and for every $B \in \NN_+$, we let $\mathscr{A}_B$ denote the set described in Proposition \ref{Shape}.

Proposition \ref{ComputationTestIdeal:P} immediately gives rise to the following simple algorithms.  

\begin{alg}[An algorithm for determining whether a candidate $F$-jumping number is in $\jnset{f}$] \
\label{AlgIsoSingTestIdeal}
\begin{description}
\item[Input] An integer bound $B$ for $\size{\jnset{f} \cap [0,1)}$ (e.g., as described in Remark \ref{RemBounds: R}), and a candidate $F$-jumping number $\lambda \in \mathscr{A}_B$.
\item[Output]  TRUE if $\lambda \in \jnset{f}$, and FALSE otherwise.

\item[Process]
Compute each of the following quantities:
\begin{enumerate}
\item $\mexp{\lambda}{p} = (u,v)$ as described in Remark \ref{mexp: R}, and  $\stop:=u+vB$.
\item $\test{f}{\tr{\lambda}{\stop}}$ and $\displaytest{f}{\tr{\lambda}{\stop} + p^{-\stop}}$ as described in Remarks \ref{ibracketcomputation: R} and \ref{stabilizationTI: R}.
\end{enumerate}
\item[Return] TRUE if $\test{f}{\tr{\lambda}{\stop}} \neq \test{f}{\tr{\lambda}{\stop} + p^{-\stop}}$, and FALSE otherwise.
\end{description}
\end{alg}

As described in Remark \ref{varyingParameter: R}, the test ideals corresponding to parameters in $[0,1)$ completely determine the set of all test ideals.  Thus, given an integer $B$ as above, we may run Algorithm \ref{AlgIsoSingTestIdeal} to check whether every candidate $F$-jumping number in $[0,1)$, i.e., every element of the finite set $\mathscr{A}_B \cap [0,1)$, is an actual $F$-jumping number, and may thus completely determine $\jnset{f} \cap [0,1)$.  Below, we describe a slightly more efficient way to compute $\jnset{f} \cap [0,1)$.

\begin{alg}[An algorithm for computing all test ideals and $F$-jumping numbers in $[0,1)$] \ 
\label{AlgIsoSingTodo} 
\begin{description}
\item[Input] An integer bound $B$ for $\size{\jnset{f} \cap [0,1)}$ (e.g., as described in Remark \ref{RemBounds: R}).
\item[Output]  The set $\jnset{f} \cap [0,1)$, as well as the test ideals $\set{ \test{f}{\lambda}: \lambda \in [0,1)}$.
\item[Process]  Order the elements of the finite set $\mathscr{A}_B \cap [0,1)$ as $0 = \lambda_0 < \lambda_1 < \ldots < \lambda_s$, and for every $1 \leq k \leq s$, compute the following quantities:
\begin{enumerate}
\item $\mexp{\lambda_k}{p} = (u_k,v_k)$ as described in Remark \ref{mexp: R}, and $\stop_k = u_k + v_k B$.
\item $\test{f}{\lambda_k} = \displaytest{f}{\tr{\lambda_k}{\stop_k} + p^{-\stop_k}}$ as described in Remarks \ref{ibracketcomputation: R} and \ref{stabilizationTI: R}.
\end{enumerate}
\item[Return] $\jnset{f}\cap [0,1) = \set{ \lambda_k :  1 \leq k \leq s \text{ and } \test{f}{\lambda_k} \neq \test{f}{\lambda_{k-1}}} \cup \set{0}$ and $\set{\test{f}{\lambda} : \lambda \in [0,1)} = \set{ \test{f}{\lambda_k} : \lambda_k \in \jnset{f} \cap [0,1)}$.
\end{description}
\end{alg}

\begin{example}
If $f = x^4 + y^3 + x^2 y^2 \in R=\FF_5 [x,y]$, then $\Length_R \quot{R}{\Jac(f)} = 6$, and using our algorithms, we find that the $F$-jumping numbers of $f$ in $(0,1)$ are $\frac{7}{12}$, $\frac{4}{5}$, and $\frac{11}{12}$; moreover, 
\begin{equation*}
\tau\( f^{\frac{7}{12}} \) = (x,y), \ \ \ 
\tau\( f^{\frac{4}{5}} \) = (x^2, y), \ \ \ \text{and} \ \ \
\tau\( f^{\frac{11}{12}} \) = (x^2, xy, y^2). 
\end{equation*}


\end{example}

We point out that for large $p$ or large bounds $B$, the algorithms given are not feasible. For more effective methods for computing $F$-pure thresholds in certain cases, we refer the reader to \cite{Diagonals,Binomials,QuasiHomog}. Additionally, we refer the reader to \cite{ST-NonPrincipal} for computations of test ideals in a more general (e.g., non-principal) setting.
\section{Local $\m$-adic constancy of $F$-pure thresholds and test ideals}
\label{madic: S}

In this section we prove our main results, Theorems \ref{ThmSC} and \ref{ThmSC-Ext}.

\begin{remark}
\label{LemDifference}
If $\mathscr{A}_B =\set{ \lambda \in \QQ_+: \exists \ (a,b) \in \expset{\lambda}{p} \text{ with } a+b \leq B} \cup \set{ 0 }$ is the set appearing in Proposition \ref{Shape}, then the difference between any two distinct elements of $\mathscr{A}_B$ is greater than $\frac{1}{p^{2B}}$, in absolute value:  If $\alpha > \beta$ are elements of $\mathscr{A}_B$, then there exist $(a,b) \in \expset{\alpha}{p}$ and $(u,v) \in \expset{\beta}{p}$ with $a+b$ and $u+v$ less than or equal to $B$.   It follows that $p^a(p^b-1)p^u(p^v-1)$ is a common denominator for $\alpha$ and $\beta$, and as $\alpha - \beta >0$, we have that
\[ \alpha - \beta \geq \frac{1}{p^a(p^b-1)p^u(p^v-1)} > \frac{1}{p^{a+b+u+v}} \geq \frac{1}{p^{2B}}.\]
\end{remark}

\begin{lemma} \label{boundSum: L}
Fix an $F$-finite polynomial ring $R$ with homogeneous maximal ideal $\m$,  a polynomial $f$ with $\sqrt{\Jac(f)} = \m$, and set $S=R_{\m}$.  If $\ell = \Length_R \quot{R}{\Jac(f)}$, and $h \in \m^{\ell+3}$, then
\begin{enumerate}
 \item $\Jac(f) \cdot S = \Jac(f+h) \cdot S$, and \label{sameJ: i}
 \item both $\size{\jnset{S, f} \cap [0,1)}$ and $\size{\jnset{S, f+h}\cap [0,1)}$ are bounded above by $\ell$.
\item In particular, $\jnset{S, f} \cup \ \jnset{S, f+h} \subseteq \mathscr{A}_{\ell}$, the set described in Proposition \ref{Shape}.  \label{sameset: i}
\end{enumerate}

\end{lemma}

\begin{proof} 
By the definition of $\ell$, we have that $\m^{\ell+1}\subseteq \Jac(f)$.  
\enote{Since $h \in \m^{\ell+3}$, } 
the polynomial $h$ and all of its partial derivatives are contained in $\m^{\ell+2}$, and hence in $\m\cdot\Jac(f)$.  It follows from this observation, the identity $f = (f+h) -h$, and the analogous identity relating partial derivatives, that 
 \[ \Jac(f) =  \Jac(f) + \m \cdot \Jac(f) = \Jac(f+h) + \m \cdot \Jac(f).\]    
The same is true after passing to the localization $S$, and it now follows from a well-known consequence\footnote{Namely, if $(A, \mathfrak{n})$ is local and $M = N + \mathfrak{n} \cdot M$ for some finitely generated $A$-modules $M$ and $N$, then $M=N$.}   of Nakayama's lemma that  $\Jac(f) \cdot S = \Jac(f+h) \cdot S$.

For every $\lambda \in [0,1)$, Remark \ref{RemBounds: R} \eqref{bounds3} and Proposition \ref{commutesWithLocalization: P} show that 
\[ \Jac(f) \cdot S \subseteq \testR{R}{f}{\lambda} \cdot S  = \testR{S}{f}{\lambda},\] and, similarly, that 
$\Jac(f) \cdot S = \Jac(f+h) \cdot S  \subseteq \testR{R}{(f+h)}{\lambda} \cdot S  = \testR{S}{(f+h)}{\lambda}$.  
These containments and Remark \ref{bijectionJNTI: R} imply that $\size{ \jnset{S,f} \cap [0,1) }$ and $\size{ \jnset{S,f+h} \cap [0,1) }$ are both bounded above by the length of $\quot{S}{\Jac(f+h) \cdot S} = \quot{S}{\Jac(f) \cdot S}$ as an $S$ module.
Finally, the last assertion follows from setting $B= \ell$ in Proposition \ref{Shape}.
\end{proof}

\begin{theorem}\label{ThmSC}
Given an $F$-finite polynomial ring $R$ with graded maximal ideal $\m$,
fix $f \in R$ such that $\sqrt{\Jac(f)} = \m$, and let $S=R_{\m}$.
If $\ell = \Length_R \quot{R}{\Jac(f)}$ and $N=p^{2 \ell} \dim(R)$, then \[\fpt{S, f} = \fpt{S, f+h} \ \text{ for every } \ h \in \m^{N}.\]
\end{theorem}

\begin{proof}
Let $\mathfrak{n} = \m S$ denote the maximal ideal of $S$.  As detailed in Remark \ref{FPTasFT: Rmk}, the $F$-pure thresholds in question may be described as the $F$-thresholds of $f$ and $f+h$ with respect to the ideal $\mathfrak{n}$, and setting $e=0$ and $g=f+h$ in Lemma \ref{Lem Ineq} 
then shows that 
\[ \left| \fpt{S,f} - \fpt{S,f+h} \right| \leq \frac{\dim(R)}{ p^{2 \ell} \dim(R)} = \frac{1}{p^{2 \ell}}.\]  
Finally, Lemma \ref{boundSum: L} \eqref{sameset: i} and Remark \ref{LemDifference} (with $B = \ell$) show that this is impossible unless these two $F$-jumping numbers are equal.
\end{proof}

In the last result of this paper, Theorem \ref{ThmSC-Ext}, we will see that increasing the value of $N$ in Theorem \ref{ThmSC} allows one to arrive at a much stronger conclusion.  In establishing this result, we will require the following lemma, which translates the issue of showing equality of test ideals to one of showing equality of various $F$-thresholds.

\begin{lemma}
\label{FTTI: L}

If $f$ and $g$ are elements of an $F$-finite regular ring such that $\ft{f}{\test{f}{\lambda}} = \ft{g}{\test{f}{\lambda}}$, then $\test{g}{\lambda} \subseteq \test{f}{\lambda}$.  In particular, if $\ft{f}{\test{f}{\lambda}} = \ft{g}{\test{f}{\lambda}}$ and $\ft{f}{\test{g}{\lambda}} = \ft{g}{\test{g}{\lambda}}$, then $\test{f}{\lambda} = \test{g}{\lambda}$.
\end{lemma}

\begin{proof}
By Proposition \ref{F-thresholds: P} \eqref{Fthres2: i},  $\ft{f}{\test{f}{\lambda}} \leq \lambda$.   This inequality accounts for the first inclusion in 
\[ \test{g}{\lambda} \subseteq \test{g}{\ft{f}{\test{f}{\lambda}}} = \test{g}{\ft{g}{\test{f}{\lambda}}} \subseteq \test{f}{\lambda},\] and the final containment follows from Proposition \ref{F-thresholds: P} \eqref{Fthres1: i}.   The last assertion then follows by interchanging $f$ and $g$ in the first.
\end{proof}

\begin{theorem}\label{ThmSC-Ext}
Given an $F$-finite polynomial ring $R$ with graded maximal ideal $\m$, fix $f \in R$ such that $\sqrt{\Jac(f)} = \m$, and let $S=R_{\m}$. 
If $\ell = \Length_R \quot{R}{\Jac(f)}$ and $N= p^{2\ell+1}(\ell+1)\dim(R)$, then  \[\testR{S}{f}{\lambda} = \testR{S}{(f+h)}{\lambda} \text{ for every $h \in \m^N$ and $0 \leq \lambda < 1$}.\]
\end{theorem}

\begin{proof}
Fix $\lambda \in [0,1)$ and $h \in \m^N$.  By Lemma \ref{FTTI: L}, to show that $\ideala:=\test{S,f}{\lambda}$ equals $\idealb:= \test{S,(f+h)}{\lambda}$, it suffices to show that 
\begin{equation}
\label{reduction: e}
\ft{f}{\ideala} = \ft{f+h}{\ideala} \text{ and } \ft{f}{\idealb} = \ft{f+h}{\idealb}.
\end{equation}

Toward establishing \eqref{reduction: e}, set $e = \up{\log_p(\ell+1)}$. It follows that $p^e \geq \ell+1$, and, therefore,  $\bracket{\m}{e} \subseteq \m^{p^e} \subseteq \m^{\ell+1} \subseteq \Jac(f)$.  In light of this, Lemma \ref{boundSum: L} \eqref{sameJ: i} implies that \[ \bracket{\m}{e} \cdot S \subseteq \Jac(f) \cdot S = \Jac(f+h) \cdot S, \] and as $\lambda \in [0,1)$, Remark \ref{RemBounds: R} \eqref{bounds3} allows us to conclude that 
\[\bracket{\m}{e} \cdot S \subseteq \Jac(f) \cdot S \subseteq \testR{R}{f}{\lambda} \cdot S = \ideala \ \text{ and } \ \bracket{\m}{e} \cdot S \subseteq \Jac(f+h) \cdot S \subseteq \testR{R}{(f+h)}{\lambda} \cdot S = \idealb.\]
Consequently, setting 
$g = f+h$ in Lemma \ref{Lem Ineq} implies that 
\begin{equation}
\label{differencetoobig2: e}
| \ft{f}{\ideala} - \ft{f+h}{\ideala} |, | \ft{f}{\idealb} - \ft{f+h}{\idealb} | \leq \frac{p^e \cdot \dim(R)}{(\ell+1) \cdot \dim(R) \cdot p^{2 \ell+1}} \leq \frac{1}{p^{2 \ell}}.
\end{equation}

On the other hand, Proposition \ref{F-thresholds: P} and Lemma \ref{boundSum: L} \eqref{sameset: i} imply that all of the $F$-thresholds under consideration are in $\jnset{S,f} \cup \jnset{S,f+h} \subseteq \mathscr{A}_{\ell}$, and setting $B=\ell$ in Remark \ref{LemDifference} then shows that the two values on the left-hand side of \eqref{differencetoobig2: e} must be zero.  We conclude that $\ft{f}{\ideala} = \ft{f+h}{\ideala}$ and $\ft{f}{\idealb} = \ft{f+h}{\idealb}$, which establishes \eqref{reduction: e}, and hence, allows us to conclude our proof.
\end{proof}

\section*{Acknowledgments}
The authors would like to thank Mel Hochster and Felipe P\'erez for  helpful discussions. 
We are grateful for support from the National Science Foundation (NSF) and the National Council of Science and Technology of Mexico (CONOCYT). 
Hern\'andez was partially supported by NSF Postdoctoral Research Fellowship DMS-1304250 and NSF Grant DMS-1600702, 
N\'u\~nez-Betancourt by CONACYT Grant 207063 and NSF Grant DMS-1502282,
and Witt by NSF Grant DMS-1623035 (DMS-1501404 before transfer).  
We would also like to thank the anonymous referee for some useful suggestions.

\bibliographystyle{amsalpha}
\bibliography{References}

\providecommand{\bysame}{\leavevmode\hbox to3em{\hrulefill}\thinspace}
\providecommand{\MR}{\relax\ifhmode\unskip\space\fi MR }
\providecommand{\MRhref}[2]{%
  \href{http://www.ams.org/mathscinet-getitem?mr=#1}{#2}
}
\providecommand{\href}[2]{#2}
\begin{thebibliography}{HNBWZ16}

\bibitem[{\`A}MBL05]{AMBL}
Josep {\`A}lvarez~Montaner, Manuel Blickle, and Gennady Lyubeznik,
  \emph{Generators of {$D$}-modules in positive characteristic}, Math. Res.
  Lett. \textbf{12} (2005), no.~4, 459--473. \MR{2155224 (2006m:13024)}

\bibitem[BFS13]{BFS2013}
Ang{{\'e}}lica Benito, Eleonore Faber, and Karen~E. Smith, \emph{Measuring
  singularities with {F}robenius: the basics}, Commutative algebra, Springer,
  New York, 2013, pp.~57--97. \MR{3051371}

\bibitem[BK14]{Alg14}
Alberto~F. Boix and Mordechai Katzman, \emph{An algorithm for producing
  {$F$}-pure ideals}, Arch. Math. (Basel) \textbf{103} (2014), no.~5, 421--433.
  \MR{3281290}

\bibitem[BMS08]{BMS-MMJ}
Manuel Blickle, Mircea Musta{\c{t}}{\v{a}}, and Karen~E. Smith,
  \emph{Discreteness and rationality of {$F$}-thresholds}, Michigan Math. J.
  \textbf{57} (2008), 43--61, Special volume in honor of Melvin Hochster.
  \MR{2492440 (2010c:13003)}

\bibitem[BMS09]{BMS-Hyp}
Manuel Blickle, Mircea Musta{\c{t}}{\u{a}}, and Karen~E. Smith,
  \emph{{$F$}-thresholds of hypersurfaces}, Trans. Amer. Math. Soc.
  \textbf{361} (2009), no.~12, 6549--6565. \MR{2538604 (2011a:13006)}

\bibitem[dFEM10]{dfEM2010}
Tommaso de~Fernex, Lawrence Ein, and Mircea Musta{\c{t}}{\u{a}},
  \emph{Shokurov's {ACC} conjecture for log canonical thresholds on smooth
  varieties}, Duke Math. J. \textbf{152} (2010), no.~1, 93--114. \MR{2643057
  (2011c:14036)}

\bibitem[dFEM11]{dfEM2011}
\bysame, \emph{Log canonical thresholds on varieties with bounded
  singularities}, Classification of algebraic varieties, EMS Ser. Congr. Rep.,
  Eur. Math. Soc., Z{\"u}rich, 2011, pp.~221--257. \MR{2779474 (2012f:14020)}

\bibitem[Her14]{Binomials}
Daniel~J. Hern{\'a}ndez, \emph{{$F$}-pure thresholds of binomial
  hypersurfaces}, Proc. Amer. Math. Soc. \textbf{142} (2014), no.~7,
  2227--2242. \MR{3195749}

\bibitem[Her15]{Diagonals}
Daniel~J. Hern\'andez, \emph{{$F$}-invariants of diagonal hypersurfaces}, Proc.
  Amer. Math. Soc. \textbf{143} (2015), no.~1, 87--104. \MR{3272734}

\bibitem[HH90]{HoHu1}
Melvin Hochster and Craig Huneke, \emph{Tight closure, invariant theory, and
  the {B}rian\c con-{S}koda theorem}, J. Amer. Math. Soc. \textbf{3} (1990),
  no.~1, 31--116. \MR{1017784 (91g:13010)}

\bibitem[HH94]{HoHu2}
\bysame, \emph{{$F$}-regularity, test elements, and smooth base change}, Trans.
  Amer. Math. Soc. \textbf{346} (1994), no.~1, 1--62. \MR{1273534 (95d:13007)}

\bibitem[HMX14]{ACCLCT}
Christopher~D. Hacon, James McKernan, and Chenyang Xu, \emph{A{CC} for log
  canonical thresholds}, Ann. of Math. (2) \textbf{180} (2014), no.~2,
  523--571. \MR{3224718}

\bibitem[HNBWZ16]{QuasiHomog}
Daniel~J. Hern\'andez, Luis {N\'{u}\~{n}ez}-Betancourt, Emily~E. Witt, and
  Wenliang Zhang, \emph{{$F$}-pure thresholds of homogeneous polynomials},
  Michigan Math. J. \textbf{65} (2016), no.~1, 57--87. \MR{3466816}

\bibitem[HY03]{H-Y}
Nobuo Hara and Ken-Ichi Yoshida, \emph{A generalization of tight closure and
  multiplier ideals}, Trans. Amer. Math. Soc. \textbf{355} (2003), no.~8,
  3143--3174 (electronic). \MR{1974679 (2004i:13003)}

\bibitem[Kat08]{Alg08}
Mordechai Katzman, \emph{Parameter-test-ideals of {C}ohen-{M}acaulay rings},
  Compos. Math. \textbf{144} (2008), no.~4, 933--948. \MR{2441251}

\bibitem[KLZ11]{KLZ}
Mordechai Katzman, Gennady Lyubeznik, and Wenliang Zhang, \emph{An upper bound
  on the number of {$F$}-jumping coefficients of a principal ideal}, Proc.
  Amer. Math. Soc. \textbf{139} (2011), no.~12, 4193--4197. \MR{2823064
  (2012k:13012)}

\bibitem[KS12]{Alg12}
Mordechai Katzman and Karl Schwede, \emph{An algorithm for computing compatibly
  {F}robenius split subvarieties}, J. Symbolic Comput. \textbf{47} (2012),
  no.~8, 996--1008. \MR{2912024}

\bibitem[MTW05]{MTW2005}
Mircea Musta{\c{t}}{\v{a}}, Shunsuke Takagi, and Kei-ichi Watanabe,
  \emph{F-thresholds and {B}ernstein-{S}ato polynomials}, European {C}ongress
  of {M}athematics, Eur. Math. Soc., Z{\"u}rich, 2005, pp.~341--364.
  \MR{MR2185754 (2007b:13010)}

\bibitem[Sho92]{ShokurovACC}
V.~V. Shokurov, \emph{Three-dimensional log perestroikas}, Izv. Ross. Akad.
  Nauk Ser. Mat. \textbf{56} (1992), no.~1, 105--203. \MR{1162635 (93j:14012)}

\bibitem[ST14]{ST-NonPrincipal}
Karl Schwede and Kevin Tucker, \emph{Test ideals of non-principal ideals:
  {C}omputations, jumping numbers, alterations and division theorems}, J. Math.
  Pures Appl. (9) \textbf{102} (2014), no.~5, 891--929. \MR{3271293}

\bibitem[TW04]{TW2004}
Shunsuke Takagi and Kei-ichi Watanabe, \emph{On {F}-pure thresholds}, J.
  Algebra \textbf{282} (2004), no.~1, 278--297. \MR{MR2097584 (2006a:13010)}

\end{thebibliography}
\vspace{.4cm}
\small{
\noindent \textsc{Department of Mathematics, University of Kansas, Lawrence, KS 66045} \\ \indent \emph{Email address}: {\tt hernandez@ku.edu}

\vspace{.3cm}

\noindent \textsc{Centro de Investigaci\'on en Matem\'aticas, Guanajuato, Gto., M\'exico} \\ \indent \emph{Email address}:  {\tt luisnub@cimat.mx} 

\vspace{.3cm}

\noindent \textsc{Department of Mathematics, University of Kansas, Lawrence, KS 66045} \\ \indent \emph{Email address}:  {\tt witt@ku.edu} 
}

\end{document}